\documentclass[12pt]{article}
\usepackage{amsmath, amssymb, amsfonts, amsthm, amscd, vmargin}
\usepackage[latin1]{inputenc}
\usepackage[all]{xy}
\usepackage{graphicx}
\usepackage{enumitem}
\usepackage{tikz}
\usetikzlibrary{shapes}
\usepackage{url}

\setmarginsrb{2.4cm}{3.3cm}{2.4cm}{3.0cm}{0cm}{0cm}{0cm}{1.5cm}

\theoremstyle{plain}
\newtheorem{teo}{Theorem}[section]
\newtheorem{lem}[teo]{Lemma}
\newtheorem{prop}[teo]{Proposition}
\newtheorem{cor}[teo]{Corollary}

\theoremstyle{definition}
\newtheorem{defi}[teo]{Definition}
\newtheorem{ex}[teo]{Example}

\newtheorem{rem}[teo]{Remark}

\numberwithin{equation}{teo}

\newcommand{\Cbb}{{\mathbb C}}
\newcommand{\Qbb}{{\mathbb Q}}
\newcommand{\Zbb}{{\mathbb Z}}

\begin{document}
\title{A survey on the singularities of spherical varieties}
\author{Boris Pasquier}
\maketitle
\abstract{We list combinatorial criteria of some singularities, which appear in the Minimal Model Program or in the study of (singular) Fano varieties, for spherical varieties. Most of the results of this paper are already known or are quite easy corollary of known results. We collect these results, we precise some proofs and add few results to get a coherent and complete survey.}
\section{Introduction}

Spherical varieties form a big family of rational and normal complex varieties, including toric varieties and flag varieties. Here, we list some types of singularities of these varieties and we give combinatorial criteria of these singularities. Most of the results of this paper are already known or are quite easy corollary of known results. \\

We will only consider types of singularities that appear in birational geometry: in the Minimal Model Program or in the study of (singular) Fano varieties.\\

The paper is organized as follows.
In section~2, we describe the theory of spherical varieties and fix notations. In section~3, we deal with smooth, locally factorial and $\Qbb$-factorial spherical varieties. In section~4, we quickly give the criteria of Gorenstein and $\Qbb$-Gorenstein singularities for spherical varieties. In section~5, we characterize terminal, canonical and kawamata log terminal (klt) singularities for spherical varieties.
And in section~6, we conclude by an overview diagram and an example.

\section{Notations}
In all the paper, $G$ denotes a connected reductive algebraic group over $\Cbb$ and varieties are algebraic varieties over $\Cbb$.\\

A spherical variety $X$ is combinatorially associated to:
\begin{itemize}
\item a lattice $M$ of characters of a Borel subgroup $B$ of $G$ (and its dual $N$);
\item a cone $\mathcal{V}$ in $N_\Qbb:=N\otimes_\Zbb\Qbb$ (called the valuation cone);
\item the finite set $\mathcal{D}$ of $B$-stable, but not $G$-stable, irreducible divisors of $X$;
\item an injective map $\sigma$ from the set of $G$-invariant valuations of $\Cbb(X)$ to $\mathcal{V}\subset N_\Qbb$ (and a map, still called $\sigma$ from $\mathcal{D}$ to $N$);
\item and a colored fan $\mathbb{F}_X$ in $N_\Qbb$.
\end{itemize}

In this section, we describe all these objects and their connections to spherical varieties, and we recall the results we will need in the next sections. The main references of the theory summarized here, are \cite{LunaVust} and \cite{Knop89} for the classification of spherical varieties, \cite{Briondiv} for propoerties of divisors of spherical varieties, \cite{Brioncurves} and \cite{Luna} for the description of anticanonical divisors of spherical varieties.

\subsection{Spherical embeddings}

\begin{defi} 
A $G$-variety $X$ is an algebraic variety over $\Cbb$ equipped with an action of $G$ on $X$.

A spherical $G$-variety is a normal $G$-variety $X$ such that there exists $x\in X$ and a Borel subgroup $B$ of $G$ satisfying that the $B$-orbit of $x$ is open in $X$.
\end{defi}

Remark that, since the Borel subgroups of $G$ are all conjugated, we can fix a Borel subgroup $B$ of $G$ and give the following equivalent definition. 

\begin{defi}
A spherical $G$-variety is a normal $G$-variety with an open $B$-orbit.
\end{defi}

When there is no possible confusion about the group acting on varieties, we write spherical varieties instead of spherical $G$-varieties.

We can easily remark that spherical varieties have an open $G$-orbit, and have finitely many $B$-orbits and $G$-orbits. The open $G$-orbit of a spherical variety $X$ is isomorphic to an homogeneous space $G/H$ that is spherical as a $G$-variety. We say that $G/H$ is a spherical homogeneous space. Note that we can always choose $H$ such that $BH$ is open in $G$.

Among spherical varieties, we can distinguish several subfamilies of varieties: flag varieties (when $H$ is a parabolic subgroup of $G$ and then $X=G/H$), toric varieties (when $G=(\Cbb^*)^n$ and $H=\{1\}$), horospherical varieties (when $H$ contains a maximal unipotent subgroup of $G$) and symmetric varieties (when $H$ is an open subgroup of the fixed points set of an involution $\sigma$ in the automorphism group of $G$).\\

The classification, in terms of colored fans, is due to D.~Luna and T.~Vust (see also \cite{Knop89}). This classification is done by fixing a spherical homogeneous space $G/H$ (ie it classifies $G/H$-embedding in terms of colored fans).

\begin{defi}
Let $G/H$ be an homogeneous space. A $G/H$-embedding is a couple $(X,x)$, where $X$ is a normal $G$-variety and $x$ is a point of $X$ such that $G\cdot x$ is open in $X$ and $H$ is the stabilizer of $x$ in $G$.

Two $G/H$-embeddings $(X,x)$ and $(X',x')$ are isomorphic if there exists a $G$-equivariant isomorphism from $X$ to $X'$ that sends $x$ to $x'$.
\end{defi}

Note that, if $G/H$ is a spherical (resp. toric, horospherical, symmetric) homogeneous space and $(X,x)$ is a $G/H$-embedding, then $X$ is a spherical (resp. toric, horospherical, symmetric) variety. Inversely, if $X$ is a spherical (resp. toric, horospherical, symmetric) $G$-variety, let $x$ be a point in the open $G$-orbit of $X$ and let $H$ be the stabilizer in $G$ of $x$; then  $G/H$ is a spherical (resp. toric, horospherical, symmetric) homogeneous space and $(X,x)$ is a $G/H$-embedding.

By abuse, we often forget the point $x$ and say that $X$ is a $G/H$-embedding if $X$ is a normal variety with an open $G$-orbit isomorphic to $G/H$.

\subsection{The lattices $M$, $N$, and the colors of $G/H$}

The field $\Cbb(X)$ of rational functions on a spherical variety $X$ can be described by a lattice $M$ by noticing the following multiplicity free property.

\begin{rem}\label{rem:multfree}
Let $X$ be a spherical $G$-variety.
The field $\Cbb(X)$ of rational functions of $X$ is naturally a $G$-module. Let $\chi$ be a character of the Borel subgroup $B$ and $f_1$, $f_2$ be in $\Cbb(X)\backslash\{0\}$ such that, for any $b\in B$, $b\cdot f_1=\chi(b)f_1$ and $b\cdot f_2=\chi(b)f_2$. Then $\frac{f_1}{f_2}$ is a rational function fixed by $B$. Since $B$ acts with an open orbit on $X$, we deduce that $\frac{f_1}{f_2}$ is a constant.
\end{rem} 

\begin{defi}\label{def:combiMNcolors}
Let $G/H$ be a spherical homogeneous space. 
\begin{enumerate}
\item We denote by $M$ the lattice of weights $\chi$ of $B$ such that there exists $f_\chi\in\Cbb(G/H)\backslash\{0\}$ satisfying, for any $b\in B$, $b\cdot f_\chi=\chi(b)f$. Note that, by Remark~\ref{rem:multfree}, for any $\chi\in M$ the rational function $f_\chi$ is unique up to a scalar. We say that $f_\chi$ is a rational function of $X$ of weight $\chi$.
\item The dual ${\rm Hom}_\Zbb(M,\Zbb)$ of $M$ is denoted by $N$.
\item We denote by $M_\Qbb$ (resp. $N_\Qbb$) the $\Qbb$-vector space $M\otimes_\Zbb\Qbb$ (resp. $N\otimes_\Zbb\Qbb$).
\item The colors of the spherical homogeneous space $G/H$ are the $B$-stable divisors of $G/H$.
\item We denote by $\mathcal{D}$ the set of colors of $G/H$.
\end{enumerate}
\end{defi}

The rank of the lattice $M$ (which is also the rank of $N$) is called the rank of $G/H$ (or the rank of $X$ if $X$ is any $G/H$-embedding).

\subsection{The valuation cone and the maps $\sigma$}
\begin{defi}\label{def:valcone}
Let $G/H$ be a spherical homogeneous space. 
 \begin{enumerate}
\item The set of $G$-invariant valuations on $\Cbb(G/H)\backslash\{0\}$ (over $\Qbb$) is a cone $\mathcal{V}$ and there is an injective map $\sigma:\mathcal{V}\longrightarrow N_\Qbb$ defined by, for any $\nu\in\mathcal{V}$ and any $\chi\in M$, $\sigma(\nu)(\chi)=\nu(f_\chi)$ where $f_\chi$ is as in Definition~\ref{def:combiMNcolors}~(1). The image of $\mathcal{V}$ in $N_\Qbb$ is called the valuation cone of $G/H$. 
\item Any color $D$ of the spherical homogeneous space $G/H$ defines a $B$-invariant valuation on $\Cbb(G/H)\backslash\{0\}$, and then it defines (similarly to the definition of $\sigma$) a point in $N_\Qbb$ that we also denote by $\sigma(D)$. It is called the image of the color $D$ in $N_\Qbb$. (In fact, $\sigma(D)\in N$.)
\end{enumerate}
\end{defi}
\begin{rem}\label{rem:valcone}
\begin{itemize}[label=$\bullet$]
\item The valuation cone of $G/H$ equals $N_\Qbb$ if and only if $G/H$ is horospherical \cite[Corollaire~5.4]{BrionPauer}.
%

\item Since the colors of a spherical homogeneous space $G/H$ are not $G$-stable, it could happen that two colors of $G/H$ have the same image in $N_\Qbb$.
\item The opposite $-\mathcal{V}^\vee$ of the dual in $M_\Qbb$ of the valuation cone is simplicial, and generated by positive roots and sums of two strongly orthogonal positive roots  \cite{BrionValCone}.
\end{itemize}
\end{rem}

\begin{defi}
The primitive elements of the rays of the simplicial cone $-\mathcal{V}^\vee$ are called the spherical roots of $G/H$.
\end{defi}

\subsection{Colored fans}

\begin{defi}
Let $G/H$ be a spherical homogeneous space, $N_\Qbb$, $\mathcal{D}$, $\mathcal{V}$ and $\sigma$ defined as above.
\begin{enumerate}
\item A colored cone in $N_\Qbb$ is a couple $(\mathcal{C},\mathcal{F})$ such that
\begin{itemize}
\item $\mathcal{F}$ is a subset of $\mathcal{D}$;
\item $\mathcal{C}$ is a cone in $N_\Qbb$ generated by finitely many elements of $N\cap\mathcal{V}$ and the $\sigma(D)$ with $D$ in $\mathcal{F}$;
\item the relative interior of $\mathcal{C}$ intersects $\mathcal{V}$;
\item $\sigma(D)\neq 0$ for any $D\in\mathcal{F}$ and $\mathcal{C}$ contains no line.
\end{itemize}
\item A colored face of a colored cone $(\mathcal{C},\mathcal{F})$ is a couple $(\mathcal{C}',\mathcal{F}')$ where $\mathcal{C}'$ is a face of the cone $\mathcal{C}$, whose relative interior intersects $\mathcal{V}$, and $\mathcal{F}'=\{D\in\mathcal{F}\,\mid\,\sigma(D)\in\mathcal{C}'\}$. It is in particular a colored cone.
\item A colored fan in $N_\Qbb$ is a finite set $\mathbb{F}$ of colored cones such that: any colored face of a colored cone of $\mathbb{F}$ is in $\mathbb{F}$, and for any $u\in N_\Qbb\cap\mathcal{V}$ there exists at most one colored cone $(\mathcal{C},\mathcal{F})$ of $\mathbb{F}$ such that $u$ is in the relative interior of $\mathcal{C}$.
\item A fan is complete if $\bigcup_{(\mathcal{C},\mathcal{F})\in\mathbb{F}}\mathcal{C}\supset\mathcal{V}$.
\end{enumerate}
\end{defi}

Note that if $G/H=(\Cbb^*)^n$, then $\mathcal{D}$ is empty (i.e. there is no color), $\mathcal{V}=N_\Qbb$, and the definitions of colored cones and colored fans are equivalent to the definitions of cones and fans in toric geometry.\\

\subsection{The classification of $G/H$-embeddings}

To any $G/H$-embedding $(X,x)$, we associate a colored fan as follows.

For any $G$-orbit $Y$ of $X$, denote by $X_Y$ the $G$-stable subset $\{x\in X\,\mid\,\overline{G\cdot x}\supset Y\}$. Denote by $\mathcal{D}_{G,Y}$ the set of $G$-stable irreducible divisors in $X_Y$ and denote by $\mathcal{F}_{Y}$ the set of $\alpha\in\mathcal{S}_P$ such that the closure of $D_\alpha$ in $X_Y$ contains $Y$. For any $D\in\mathcal{D}_{G,Y}$, denote by $\sigma(D)$ the image by $\sigma$ of the $G$-invariant valuation on $\Cbb(G/H)\backslash\{0\}$ associated to $D$. Denote by $\mathcal{C}_Y$ the cone in $N_\Qbb$ generated by the $\sigma(D)$ with $D\in\mathcal{D}_{G,Y}$ and the $\sigma(D)$ with $D\in\mathcal{F}_{Y}$. 

We can now state the classification of $G/H$-embeddings when $G/H$ is spherical.

\begin{teo}[\cite{Knop89}] Let $G/H$ be a spherical homogeneous space.

Let $(X,x)$ be a $G/H$-embedding.
Then, for any $G$-orbit $Y$ of $X$, $(\mathcal{C}_Y,\mathcal{F}_{Y})$ is a colored cone in $N_\Qbb$ and the set of $(\mathcal{C}_Y,\mathcal{F}_{Y})$ with $Y$ in the set of $G$-orbits of $X$ is a colored fan in $N_\Qbb$. It is called the colored fan of $X$ and denoted by $\mathbb{F}_X$.

The map from the set of isomorphic classes of $G/H$-embeddings to the set of colored fans in $N_\Qbb$ that sends the class of $(X,x)$ to $\mathbb{F}_X$ is well-defined and bijective. 

Moreover, $X$ is complete if and only if $\mathbb{F}_X$ is complete.
\end{teo}

The set of colors of $X$ (or of $\mathbb{F}_X$) is the union $\mathcal{F}_X:=\bigcup_{(\mathcal{C},\mathcal{F})\in\mathbb{F}}\mathcal{F}$. It is a subset of $\mathcal{D}$.

\subsection{$G$-equivariant morphisms between $G/H$-embeddings}

The $G$-equivariant morphisms between spherical $G$-varieties are very well understood. We state here the description of birational $G$-equivariant morphisms (ie between two embeddings of the same spherical homogeneous space.

Note that if $f:\,X\longrightarrow Y$ is a proper morphism between algebraic varieties such that $\phi_*(\mathcal{O}_X)=\mathcal{O}_Y$, and if $G$ acts on $X$, by a result of Blanchard (see also \cite[Prop. 4.2.1]{BrionSamuelUma}), then there exists a unique action of $G$ on $Y$ such that $f$ is $G$-equivariant. In particular, it is not so restrictive to only consider $G$-equivariant morphisms.

\begin{prop}\label{prop:morph}
Let $G/H$ be a spherical homogeneous space.

Let $(X,x)$, $(Y,y)$ be two $G/H$-embeddings.
Then, there exists a $G$-equivariant morphism $f:\,X\longrightarrow Y$ with $f(x)=y$ if and only if for any colored cone $(\mathcal{C},\mathcal{F})\in\mathbb{F}_X$, there exists a colored cone $(\mathcal{C}',\mathcal{F}')\in\mathbb{F}_Y$ such that $\mathcal{C}\subset\mathcal{C}'$ and $\mathcal{F}\subset\mathcal{F'}$. 
\end{prop}

\subsection{Divisors, Cartier divisors}\label{section:div}

Let $G/H$ be a spherical homogeneous space.

Let $(X,x)$ be a $G/H$-embedding associated to the colored fan $\mathbb{F}_X$. Denote by $X_1,\dots,X_m$ the irreducible $G$-stable divisors of $X$ ($m\geq 0$).
For any $i\in\{1,\dots,m\}$, we denote by $x_i$ the image by $\sigma$ of the valuation associated to $X_i$. Recall that it is a primitive element of an edge of $\mathbb{F}_X$ that has no color (i.e. contains no $\sigma(D)$ with $D\in\mathcal{F}_X$).

\begin{prop}\label{prop:Cartiercrit}
\begin{enumerate}
\item Any Weil divisor of $X$ is linearly equivalent to a $B$-stable divisor, i.e. of the form $\delta=\sum_{i=1}^ma_iX_i+\sum_{D\in\mathcal{D}}a_D D$.
\item Such a divisor is Cartier if and only if for any $(\mathcal{C},\mathcal{F})\in\mathbb{F}_X$ there exists $\chi\in M$ such that for any $x_i\in\mathcal{C}$, $\langle\chi,x_i\rangle=a_i$ and for any $D\in\mathcal{F}$, $\langle\chi,\sigma(D)\rangle=a_D$.
\end{enumerate}
\end{prop}

Then any Cartier divisor $\delta=\sum_{i=1}^ma_iX_i+\sum_{D\in\mathcal{D}}a_D D$ of a complete spherical variety $X$ is associated to a unique piecewise linear function $h_\delta$ on $\mathbb{F}_X$ (well-defined function on $\mathcal{V}$), linear on each cone in $\mathbb{F}_X$, such that $\forall i\in\{1,\dots,m\}$, $h_\delta(x_i)=a_i$ and $\forall D\in\mathcal{F}_X$, $h_\delta(\sigma(D))=a_D$. In that case, for any maximal cone $\mathcal{C}$ of $\mathbb{F}_X$ (i.e. for any maximal colored cone $(\mathcal{C},\mathcal{F})$), we denote by $\chi_{\mathcal{C},\delta}$ the element of $M$ associated to the linear function defining $h_\delta$ on $\mathcal{C}$.

\begin{defi}
A piecewise linear function $h_\delta$ on $\mathbb{F}_X$ is convex if:
\begin{itemize}
\item for any maximal cone $\mathcal{C}$ of $\mathbb{F}_X$ and any $x\in N_\Qbb$, we have $h_\delta(x)\geq\langle\chi_{\mathcal{C},\delta},x\rangle$;
\item and for any $D\in\mathcal{D}\backslash\mathcal{F}_X$, $h_\delta(\sigma(D))\leq a_D$.
\end{itemize}

We say that it is strictly convex if:

\begin{itemize}
\item  for any maximal cone $\mathcal{C}$ of $\mathbb{F}_X$ and any $x\in N_\Qbb\backslash \mathcal{C}$, we have $h_\delta(x)>\langle\chi_{\mathcal{C},\delta},x\rangle$;
\item and for any $D\in\mathcal{D}\backslash\mathcal{F}_X$, $h_\delta(\sigma(D))< a_D$.
\end{itemize}
\end{defi}

\begin{prop}\label{prop:amplecrit}
A Cartier divisor $\delta=\sum_{i=1}^ma_iX_i+\sum_{D\in\mathcal{D}}a_D D$ of a complete spherical variety $X$ is globally generated (respectively ample) if and only if $h_\delta$ is convex (respectively strictly convex).
\end{prop}

Anticanonical divisors of spherical varieties are described in \cite[Propostion~4.1]{Brioncurves} and in \cite[Section~3.6]{Luna}, by dividing the colors of $G/H$ into three types. The description of these three types was recently reconsider in \cite{Knop14} in a more simple way by using that $B$-orbits of $G/H$ correspond to $H$-orbits in $B\backslash G$.  In the following theorem, we gather together their results that we will need later.

\begin{teo}\label{th:canonicaldiv}
Let $G/H$ be a spherical homogeneous space. Let $D\in\mathcal{D}$.

Choose a simple root $\alpha$ of $G$, such that $P_\alpha\cdot D\neq D$.

Recall that the spherical roots of $G/H$ are the primitive elements of $-\mathcal{V}^\vee$.
Then one and only one of the following case occurs.

\begin{itemize}
\item[(a)] $\alpha$ is a spherical root of $G/H$.
\item[(2a)] $2\alpha$ is a spherical root of $G/H$.
\item[(b)] neither $\alpha$ nor $2\alpha$ is a spherical root of $G/H$.
\end{itemize}

Moreover, the case occurring does not depend on the choice of $\alpha$. Then, we say that $D$ is of type (a), (2a) and (b) respectively.\\

Denote by $\alpha^\vee_M$ the restriction to $M$ of the coroot $\alpha^\vee$; it is an element of $N$.

Then, the images $\sigma(D)$ of $D$ in $N_\Qbb$ satisfy, respectively in each case:

\begin{itemize}
\item[(a)] $\langle\alpha,\sigma(D)\rangle=1$;
\item[(2a)] $\sigma(D)=\frac{1}{2}\alpha^\vee_M$, in particular $\langle\alpha,\sigma(D)\rangle=1$;
\item[(b)] $\sigma(D)=\alpha^\vee_M$.
\end{itemize}

Denote by $P$ the stabilizer in $G$ of the open $B$-orbit of $G/H$ (it is a parabolic subgroup of $G$ containing $B$), and denote by $\mathcal{S}_P$ the set of simple roots $\alpha$ of $(G,B,T)$ such that $-\alpha$ is not a weight of the Lie algebra of $P$).

For any $D\in\mathcal{D}$, we define an integer $a_D$ as follows: if $D$ is of type (a) or (2a), $a_D=1$; and if $D$ is of type (b), $a_D=\langle\sum_{\alpha\in \mathcal{R}^+_{P}}\alpha,\alpha^\vee\rangle$ (which is greater or equal to 2), where $\mathcal{R}^+_{P}$ is the set of positive roots with at least one non-zero coefficient for a simple root of $\mathcal{S}_P$.\\

Let $(X,x)$ be a $G/H$-embedding. Denote by $\mathcal{D}_X$ the set of irreducible $G$-stable divisors of $X$.
Then, an anticanonical divisor of a $G$-spherical embedding $(X,x)$ associated to a colored fan $\mathbb{F}_X$ is $$-K_X=\sum_{D\in\mathcal{D}_X}D+\sum_{D\in \mathcal{D}}a_D D.$$
\end{teo}

\begin{cor}\label{cor:canonicaldiv}
For any $B$-stable irreducible divisor $D$ of $X$, the coefficient attached to $D$ in $-K_X$ a positive integer. Moreover, if $D$ is not $G$-stable and $a_D=1$, then $\sigma(D)$ is not in the valuation cone $\mathcal{V}$.
\end{cor}
\begin{proof}
If $a_D=1$, then $D$ is of type {\it (a)} or {\it (2a)}. In particular, $\alpha$ or $2\alpha$ is in $-\mathcal{V}^\vee$.
 Hence, for any $v\in \mathcal{V}\subset N$, we get $\langle\alpha,v\rangle\leq 0$. But we also have that $\langle\alpha,\sigma(D)\rangle=1$. Thus $\sigma(D)$ is not in $\mathcal{V}$.
\end{proof}

\begin{rem}
If $\sigma(D)$ is not in $\mathcal{V}$, $a_D$ is not necessary 1. For example, if $G=\rm{SL}_2\times\rm{SL}_2$ and $H$ is the diagonal in $G$, then $G/H$ has only one spherical root that is the strongly orthogonal sum of the two simple roots; in particular, the unique $B$-stable divisor $D$ of $G/H$ is of type {\it (b)} and an easy computation gives $a_D=2$.
\end{rem}

\section{Smooth, locally factorial and $\Qbb$-factorial varieties}

\begin{defi}
A variety $X$ is locally factorial if all Weil divisors of $X$ are Cartier.

A variety $X$ is $\Qbb$-factorial if all Weil divisors of $X$ are $\Qbb$-Cartier.
\end{defi}

The smoothness of spherical varieties is the type of singularities that is the most complicated to characterize for spherical varieties.
For toric varieties, it is not difficult and well-known (see for example \cite{Fulton}). For horospherical varieties, there is a more complicated criterion, simultaneously obtained in \cite{these} and \cite{Timashev}, which mixes the combinatorial aspects of colored fans and root systems. For general spherical varieties, a smoothness criterion was first given in~\cite{Brionsmooth}, and a more practical one was recently given in \cite{Gagliardi}. Moreover, if we admit a conjecture (satisfied by horospherical varieties and symmetric varieties), then we get a very simple smoothness criterion \cite{Gagliardi-Hofscheier2}.

In this paper, we will not write these smoothness criteria, but we have to note that the main tool of their proofs is the local structure of spherical varieties. This tool easily permits to prove the following useful (and well-known) result.

\begin{prop}\label{prop:toroidalsmooth}
Let $X$ be a locally factorial spherical variety such that $\mathcal{F}_X$ is empty. Then $X$ is smooth.
\end{prop}

\begin{proof}
Let $G/H$ be a spherical homogeneous space, and $(X,x)$ be a $G/H$-embedding.
Denote by $P$ the stabilizer in $G$ of the open $B$-orbit of $G/H$ (it is a parabolic subgroup of $G$ containing $B$).
Then, by \cite[Proposition~3.4]{BrionPauer}, a $P$-stable open set $\mathcal{U}$ of $X$ is isomorphic to $R_u(P)\times Z$, where $Z$ is a toric variety under the action of the neutral component of a Levi of $P$. Moreover, this open  
set $\mathcal{U}$ equals $X\backslash\bigcup_{D\in\mathcal{D}}\overline{D}$ (where $\overline{D}$ is the closure of $D$ in $X$). If $\mathcal{F}_X$ is empty, then $\mathcal{U}$ intersects every closed $G$-orbit.
If $X$ is locally factorial, then $Z$ is a locally factorial toric variety and so $Z$ is smooth. We conclude that, if $X$ is locally factorial with $\mathcal{F}_X$ empty, $X$ is smooth along all its closed $G$-orbits. Hence $X$ is smooth.
\end{proof}

We now write and prove locally factorial and $\Qbb$-factorial criteria.

Using the criterion of Cartier divisors of spherical varieties (Proposition~\ref{prop:Cartiercrit}), we get the following result.

\begin{prop}\label{prop:factoriality}
Let $X$ be a spherical variety associated to a colored fan $\mathbb{F}_X$.
Then $X$ is locally factorial (respectively $\Qbb$-factorial) if and only if for any $(\mathcal{C},\mathcal{F})\in\mathbb{F}_X$, the colors of $\mathcal{F}$ have distinct images in $N_\Qbb$ and there exists a basis $(u_1,\dots,u_k)\cup(\sigma(D))_{D\in\mathcal{F}}$ of the lattice $N$ (respectively of the vector space $N_\Qbb$) such that $\mathcal{C}$ is generated by the family $(u_1,\dots,u_{k'})\cup(\sigma(D))_{D\in\mathcal{F}}$ (where $k'\leq k$ are non-negative integers).
\end{prop}

\begin{proof}
The "if" part is quite easy, so we prove only the "only if" part.

Suppose that $X$ is $\Qbb$-factorial. Let $(\mathcal{C},\mathcal{F})\in\mathbb{F}_X$, and $D_1$, $D_2$ be two colors in $\mathcal{F}$. The Weil divisor $\delta=D_1$ is $\Qbb$-Cartier, so that there exists $\chi_{\mathcal{C},\delta}\in M_\Qbb$ such that $\langle\chi_{\mathcal{C},\delta},\sigma(D_1)\rangle=1$ and $\langle\chi_{\mathcal{C},\delta},\sigma(D_2)\rangle=0$. In particular, $\sigma(D_1)\neq\sigma(D_2)$. 
Now, denote by $u_1,\dots,u_{k'}$ the primitive elements of the edges of $\mathcal{C}$ that are not generated by some $\sigma(D)$ with $D\in\mathcal{F}$. We can suppose that $u_1,\dots,u_{k'}$ respectively correspond to the $G$-stable irreducible divisors $X_1,\dots,X_{k'}$.

For any $i\in\{1,\dots,k'\}$, $X_i$ is $\Qbb$-Cartier, so there exists $\chi_{\mathcal{C},X_i}\in M_\Qbb$ such that $\langle\chi_{\mathcal{C},X_i},u_i\rangle=1$, $\langle\chi_{\mathcal{C},X_i},u_j\rangle=0$, for any $j\in\{1,\dots,k'\}\backslash\{i\}$, and $\langle\chi_{\mathcal{C},X_i},\sigma(D)\rangle=0$, for any $D\in\mathcal{F}$. Similarly, for any $D'\in\mathcal{F}$, there exists $\chi_{\mathcal{C},D'}\in M_\Qbb$ such that $\langle\chi_{\mathcal{C},D'},\sigma(D')\rangle=1$, $\langle\chi_{\mathcal{C},D'},u_j\rangle=0$, for any $j\in\{1,\dots,k'\}$, and $\langle\chi_{\mathcal{C},D'},\sigma(D)\rangle=0$, for any $D\in\mathcal{F}\backslash\{D'\}$.

Let $a_1,\dots,a_{k'}$ and, for any $D\in\mathcal{F}$, $a_D$ be rational numbers such that $a_1u_1+\cdots+a_{k'}u_{k'}+\sum_{D\in\mathcal{F}}a_D\sigma(D)=0$ (*). 
Applying $\chi_{\mathcal{C},X_i}$ for $i\in\{1,\dots,k'\}$ and $\chi_{\mathcal{C},D'}$ for $D'\in\mathcal{F}$ to (*), we get that $a_1=\cdots=a_{k'}=0$ and $a_D=0$ for any $D\in\mathcal{F}$. The family $(u_1,\dots,u_{k'})\cup(\sigma(D))_{D\in\mathcal{F}}$ is linearly independent (and generates $\mathcal{C}$), in particular we can complete it to get a basis $(u_1,\dots,u_k)\cup(\sigma(D))_{D\in\mathcal{F}}$ of $N_\Qbb$.

Suppose moreover that $X$ is locally factorial. Then, we can choose the elements $\chi_{\mathcal{C},X_i}$ for $i\in\{1,\dots,k'\}$, and $\chi_{\mathcal{C},D'}$ for $D'\in\mathcal{F}$, in the lattice $M$. With the same proof as above, we can prove that the family $(\chi_{\mathcal{C},X_1},\dots,\chi_{\mathcal{C},X_{k'}})\cup(\chi_{\mathcal{C},D})_{D\in\mathcal{F}})$ is linearly independent.
Also, for any element $u$ in the intersection of $N$ with the $\Qbb$-vector space generated by $(u_1,\dots,u_{k'})\cup(\sigma(D))_{D\in\mathcal{F}}$, we have $u=\langle\chi_{\mathcal{C},X_1},u\rangle u_1+\cdots+\langle\chi_{\mathcal{C},X_{k'}},u\rangle u_{k'}+\sum_{D\in\mathcal{F}}\langle\chi_{\mathcal{C},D},u\rangle$ and then $u$ is in the sublattice generated by $(\chi_{\mathcal{C},X_1},\dots,\chi_{\mathcal{C},X_{k'}})\cup(\chi_{\mathcal{C},D})_{D\in\mathcal{F}})$. We conclude, by Lemma~\ref{lem:latticebasis}, that we can find $u_{k'+1},\dots,u_k$ in $N$ such that $(u_1,\dots,u_k)\cup(\sigma(D))_{D\in\mathcal{F}}$ is a basis of $N$.
\end{proof}

A particular case of the structure theorem for finitely generated free modules over a principal ideal domain, 
gives the following result.
\begin{lem}\label{lem:latticebasis}
Let $L$ be a lattice. Let $\mathcal{E}$ be a linearly independent family of elements of $L$. Denote by $L'$ the sublattice generated by $\mathcal{E}$ and suppose that $L'$ equals the intersection of $L$ with the $\Qbb$-vector space generated by $\mathcal{E}$.

Then we can complete $\mathcal{E}$ into a basis of $L$.
\end{lem}

\begin{rem}\label{rem:desing}
A consequence of Propositions~\ref{prop:morph}, \ref{prop:toroidalsmooth} and~\ref{prop:factoriality} is that, for any spherical variety $X$, a $G$-equivariant resolution $f:\,V\longrightarrow X$ is given by erasing all colors in $\mathbb{F}_X$ and by sufficiently subdividing the cones in $\mathbb{F}_X$. (Such subdivision exists for example by \cite[Theorem~1.5]{Cox}.)

Moreover, the exceptional locus of $f$ is $G$-stable. But, still Propositions~\ref{prop:toroidalsmooth} and~\ref{prop:factoriality}, $G$-stable irreducible subvarieties of $V$ are smooth. Then, by blowing-up the irreducible component of the exceptional locus of $f$ of codimension at least~2, we obtain a $G$-equivariant resolution $\tilde{f}:\,\tilde{V}\longrightarrow X$ such that the exceptional locus of $\tilde{f}$ is of pure codimension one. Remark also that, since $\mathcal{F}_{\tilde{V}}$ is empty, the exceptional divisors of $\tilde{V}$, which are $G$-stable, are smooth.

By the local structure used in Proposition~\ref{prop:toroidalsmooth}, we can also prove that the exceptional locus of $\tilde{f}$ is a simple normal crossing divisor.
\end{rem}

\section{Gorenstein and $\Qbb$-Gorenstein varieties}

\begin{defi}
A normal variety $X$ is Gorenstein (respectively $\Qbb$-Gorenstein) if the anticanonical divisor $-K_X$ is Cartier (respectively $\Qbb$-Cartier).
\end{defi}

The criterion of these types of singularities for spherical varieties, is an easy consequence of the criterion of Cartier divisors of spherical varieties (Proposition~\ref{prop:Cartiercrit}) and the description of anticanonical divisor (Theorem~\ref{th:canonicaldiv}).

\begin{prop}
Let $G/H$ be a spherical homogeneous space. For any color $D\in\mathcal{D}$, we define $a_D$ as in Section~\ref{section:div}.
Let $(X,x)$ be a $G/H$-embedding associated to a colored fan $\mathbb{F}_X$. 
Then $X$ is Gorenstein (respectively $\Qbb$-Gorenstein) if and only if for any $(\mathcal{C},\mathcal{F})\in\mathbb{F}_X$, there exists $m_\mathcal{C}\in M$ (respectively $m_\mathcal{C}\in M_\Qbb$) such that, for any primitive element $x$ of an edge of $\mathcal{C}$ that is not generated by some $\sigma(D)$ with $D\in \mathcal{D}$, $\langle m_\mathcal{C},x\rangle=1$, and for any $D\in\mathcal{F}$, $\langle m_\mathcal{C},\sigma(D)\rangle=a_D$.
\end{prop}

\section{(log) terminal and canonical singularities}

\begin{defi}
Let $X$ be a normal $\Qbb$-Gorenstein variety. Let $f:\,V\longrightarrow X$ be a resolution of $X$ (i.e. $f$ is birational and $V$ is smooth).
Then $K_V-f^*(K_X)=\sum_{i\in\mathcal{I}}a_iE_i$ where $\{E_i\,\mid\,i\in\mathcal{I}\}$ is the set of exceptional divisors of $f$.

We say that $X$ has 
\begin{itemize}[label=$\bullet$]
\item canonical singularities if, for any $i\in\mathcal{I}$, $a_i\geq 0$;
\item terminal singularities if, for any $i\in\mathcal{I}$, $a_i>0$.
\end{itemize}
\end{defi}

Note that the definition does not depend on the choice of the resolution.
Moreover, if $X$ is spherical, recall that, by Remark~\ref{rem:desing}, we can construct a resolution by deleting the colors of $X$ and by taking subdivision of the cones of $\mathbb{F}_X$. Then, still with the criterion of Cartier divisors on spherical varieties, we get the following characterizations of canonical and terminal singularities.

\begin{prop}\label{prop:canterm}
Let $G/H$ be a spherical homogeneous space. Let $(X,x)$ be a $\Qbb$-Gorenstein $G/H$-embedding  associated to a colored fan $\mathbb{F}_X$. 

For any colored cone $(\mathcal{C},\mathcal{F})$ of $\mathbb{F}_X$, denote by $h_\mathcal{C}$ the linear function such that for any $D\in\mathcal{F}$, $h_\mathcal{C}(\sigma(D))=a_D$ and, for any primitive element $u$ of an edge of $\mathcal{C}$ that is not generated by some $\sigma(D)$ with $D\in\mathcal{F}$, $h_\mathcal{C}(u)=1$.
\begin{itemize}[label=$\bullet$]
\item $X$ has canonical singularities if and only if for any colored cone $(\mathcal{C},\mathcal{F})$ of $\mathbb{F}_X$, for any $x\in\mathcal{C}\cap N\cap\mathcal{V}$, $h_\mathcal{C}(x)\geq 1$.
\item $X$ has terminal singularities if and only if for any colored cone $(\mathcal{C},\mathcal{F})$ of $\mathbb{F}_X$, for any $x\in\mathcal{C}\cap N\cap\mathcal{V}$, $h_\mathcal{C}(x)\leq 1$ implies that $x$ is the primitive element of an edge of $\mathcal{C}$ that is not generated by some $\sigma(D)$ with $D\in\mathcal{F}$ (i.e. $x=x_i$ with our notation).
\end{itemize}
\end{prop}

We begin by proving the following result.

\begin{lem}\label{lem:pullback}
Let $G/H$ be a spherical homogeneous space. Let $(X,x)$ and $(V,v)$ be two $G/H$-embeddings respectively associated to colored fans $\mathbb{F}_X$ and $\mathbb{F}_V$. Suppose that there exists a $G$-equivariant dominant morphism $f:\,V\longrightarrow X$. And let $\delta$ be a Cartier divisor of $X$.

Recall that $\delta$ is associated to a piecewise linear function $h_\delta$ on $\mathbb{F}_X$. Then, 
for any colored cone $(\mathcal{C},\mathcal{F})$ of $\mathbb{F}_X$, we denote by $h_{\mathcal{C},\delta}$ a linear function associated to the restriction of $h_\delta$ on $(\mathcal{C},\mathcal{F})$.
Moreover, this linear function can be chosen to be identified to an element $\chi_{\mathcal{C},\delta}$ of $M$.

Let $E$ be an exceptional divisor of $f$. Since $f$ is birational and $G$-equivariant, it is $G$-stable.

Then, for any $(\mathcal{C},\mathcal{F})$ in $\mathbb{F}_X$ such that $x_{E}\in\mathcal{C}$, the coefficient attached to $E$ in $f^*(\delta)$ is $h_{\mathcal{C},\delta}(\sigma(E))$.
\end{lem}

\begin{proof}[Proof of Lemma~\ref{lem:pullback}]
To describe $f^*(\delta)$ we need to look deeper at the Cartier criterion.

Since $\delta$ is a $B$-stable Cartier divisor of $X$,  $\mathcal{O}(\delta)$ is given by $(\mathcal{U}_\mathcal{C},f_{-\chi_{\mathcal{C},\delta}})$, where $(\mathcal{C},\mathcal{F})$ run the set of maximal colored cones in $\mathbb{F}_X$, where $\mathcal{U}_\mathcal{C}$ is the open set of $X$ associated to the colored cone $(\mathcal{C},\mathcal{F})$ and $f_{ -\chi_{\mathcal{C},\delta}}$ is a non-zero rational function on $X$ associated to the weight $-\chi_{\mathcal{C},\delta}$ (unique up to a scalar). Moreover, for any maximal colored cones $(\mathcal{C},\mathcal{F})$ in $\mathbb{F}_X$, $\rm{div}(f_{ -\chi_{\mathcal{C},\delta}})$ equals $-\delta$ on the open set $\mathcal{U}_\mathcal{C}$. Then $f^*\mathcal{O}(\delta)$ is given by $(f^{-1}(\mathcal{U}_\mathcal{C}),f_{-\chi_{\mathcal{C},\delta}}\circ f)$.

Now, we remark that the map from $\Cbb(X)$ to $\Cbb(V)$ that sends $g$ to $g\circ f$ is a $G$-equivariant isomorphism, so that $f_{-\chi_{\mathcal{C},\delta}}\circ f$ is a rational function on $V$ associated to the weight $-\chi_{\mathcal{C},\delta}$. And we also notice that $f^{-1}(\mathcal{U}_\mathcal{C})$ is the union of the open set of $V$ associated to the colored cones $(\mathcal{C}',\mathcal{F}')$ of $\mathbb{F}_V$ such that $\mathcal{C}'\subset\mathcal{C}$ (and $\mathcal{F}'\subset\mathcal{F}$). Then, for any irreducible $G$-stable divisor $E$ of $V$, the coefficient attached to $E$ in $f^*(\delta)$ is $\langle \chi_{\mathcal{C},\delta},\sigma(E)\rangle$ for any maximal $(\mathcal{C},\mathcal{F})\in\mathbb{F}_X$ such that $\sigma(E)\in\mathcal{C}$. 

Hence, for any $(\mathcal{C},\mathcal{F})\in\mathbb{F}_X$ such that $\sigma(E)\in\mathcal{C}$, the coefficient attached to $E$ in $f^*(\delta)$ is $h_{\mathcal{C},\delta}(\sigma(E))$.
\end{proof}

\begin{proof}[Proof of Proposition~\ref{prop:canterm}]

Let $f:\,V\longrightarrow X$ be a $G$-equivariant resolution of $X$. We can apply Lemma~\ref{lem:pullback}, the Cartier divisor $\delta=-kK_X$ for a large enough positive integer $k$. In that case, for any $(\mathcal{C},\mathcal{F})\in\mathbb{F}_X$, $h_{\mathcal{C},\delta}=kh_\mathcal{C}$. Then, for any $i\in\mathcal{I}$ and for any $(\mathcal{C},\mathcal{F})\in\mathbb{F}_X$ such that $\sigma(E_i)\in\mathcal{C}$, the coefficient attached to $E_i$ in $k(K_V-f^*(K_X))$ is $kh_{\mathcal{C}}(x_{E})-k$, so that the coefficient $a_i$ attached to $E_i$ in $K_V-f^*(K_X)$ is $h_{\mathcal{C}}(\sigma(E_i))-1$.

\begin{itemize}[label=$\bullet$]
\item Suppose that $X$ has canonical singularities. Let $(\mathcal{C},\mathcal{F})\in\mathbb{F}_X$ and $x\in\mathcal{C}\cap N\cap\mathcal{V}$. We can suppose that $x$ is primitive in $N$. By Remark~\ref{rem:desing}, there exists a resolution $V$ of $X$ such that $x$ is the primitive element of an edge without color of the colored fan of $V$. Then $x$ is the image by $\sigma$ of an exceptional ($G$-stable and irreducible) divisor $E_i$ of $V$. Since $a_i\geq 0$ and $a_i=h_\mathcal{C}(x)-1$, we get $h_\mathcal{C}(x)\geq 1$.

The "if" part proof works with the same arguments.

\item The proof is almost the same as above with "$>$" instead of "$\geq$".
\end{itemize}
\end{proof}

\begin{defi}
Let $X$ be a normal variety and let $D$ be an effective $\Qbb$-divisor such that $K_X+D$ is $\Qbb$-Cartier.
The pair $(X,D)$ is said to be klt (Kawamata log terminal) if for any resolution  $f:\,V\longrightarrow X$ of $X$ such that $K_V=f^*(K_X+D)+\sum_{i\in \mathcal{I}}a_iE_i$, we have $a_i>-1$ for any $i\in \mathcal{I}$.

We say that $X$ has log terminal singularities if $X$ is $\Qbb$-Gorenstein and $(X,0)$ is klt.
\end{defi}

\begin{rem}\label{rem:klt}
\begin{enumerate}
\item In fact, it is enough to check the above property for one log-resolution to say that a pair $(X,D)$ is klt.
\item The condition "$a_i>-1$ for any $i\in \mathcal{I}$" can be replaced by: $\lfloor D\rfloor=0$ and for any $i\in \mathcal{I}$ such that $E_i$ is exceptional for $f$, $a_i>-1$.
\end{enumerate}
\end{rem}

Still with the criterion of Cartier divisors of spherical varieties (Proposition~\ref{prop:Cartiercrit}), we get the following result.

\begin{prop}\cite{AlexeevBrion}
Let $X$ be a spherical variety. Then $X$ has log terminal singularities.
\end{prop}

In fact, in \cite{AlexeevBrion}, V.~Alexeev and M.~Brion proved that, if $X$ is a spherical $G$-variety and $D$ be an effective $\Qbb$-divisor of $X$ such that $D+K_X$ is $\Qbb$-Cartier, $\lfloor D\rfloor=0$ and $D=D_G+D_B$ where $D_G$ is $G$-stable and $D_B$ is stable under the action of a Borel subgroup $B$ of $G$, then $(X,D_G+D_B')$ has klt singularities for general $D_B'$ in $|D_B|$.

We can give a short proof of the proposition.
\begin{proof}
Let $f:\,V\longrightarrow X$ be $G$-equivariant log-resolution of $(X,0)$, ie such that the exceptional locus of $f$ is a simple normal crossing divisor (see Remark~\ref{rem:desing}). With the same arguments as in the first part of the proof of Proposition~\ref{prop:canterm}, we can prove that, for any exceptional divisor $E_i$ of $f$, we have $a_i=h_{\mathcal{C},\delta}(\sigma(E_i))-1$ where $\sigma(E_i)\in\mathcal{C}$ and $\delta=-K_X$.

By Proposition~\ref{th:canonicaldiv}, we notice that $h_{\mathcal{C},\delta}(x)>0$ for any $x\in\mathcal{C}\backslash\{0\}$. In particular $h_{\mathcal{C},\delta}(\sigma(E_i))>0$ and then $a_i>-1$.
\end{proof}

To complete what we know on klt singularities of spherical varieties, we note that the author prove in \cite{KltHoro} the following result, by using Bott-Samelson resolutions of flag varieties.

\begin{teo}
Let $X$ be a horospherical variety and let $D$ be an effective $\Qbb$-divisor such that $K_X+D$ is $\Qbb$-Cartier. The pair $(X,D)$ is klt if and only if $\lfloor D\rfloor=0$.
\end{teo}

The author does not know if this result could be generalized to spherical varieties.

\section{Conclusion and example}

To complete the natural connections between these singularities we can also state the following result.

\begin{prop}\label{prop:implication}
Any locally factorial spherical variety has terminal singularities.

And any Gorenstein spherical variety has canonical singularities.
\end{prop}

\begin{rem}
By \cite[Corollary 5.24]{KollarMori}, any variety with rational and Gorenstein singularities have canonical singularities. But, spherical varieties are rational, then the second assertion of Proposition~\ref{prop:implication} is already known.
\end{rem}

\begin{proof}

Fix a spherical homogeneous space $G/H$.

Let $X$ be a locally factorial $G/H$-embedding and let $f:\,V\longrightarrow X$ be a $G$-equivariant resolution of $X$. Let $E_i$ be an exceptional divisor of $f$. And let $(\mathcal{C},\mathcal{F})\in\mathbb{F}_X$ such that $\sigma(E)\in\mathcal{C}$.
By Proposition~\ref{prop:factoriality}, the elements of $\mathcal{F}$ have distinct images in $N_\Qbb$ and there exists a basis $(u_1,\dots,u_k)\cup(\sigma(D))_{D\in\mathcal{F}}$ of the lattice $N$ such that $\mathcal{C}$ is generated by the family $(u_1,\dots,u_{k'})\cup(\sigma(D))_{D\in\mathcal{F}}$ (where $k'\leq k$ are non-negative integers). 
In particular, $\sigma(E_i)=\sum_{j=1}^{k'}\lambda_iu_i+\sum_{D\in\mathcal{F}}\lambda_D\sigma(D)$, where the $\lambda_i$'s and the $\lambda_D$'s are non-negative integers. Moreover, either at least two of these integers are not zero, or only one $\lambda_D$ is not zero, because $\sigma(E_i)$ is a primitive element different from the $u_i$'s. In the second case, we must have $\sigma(D)$ in $\mathcal{V}$ so that, by Corollary~\ref{cor:canonicaldiv}, $a_D\geq 2$.

Hence, $h_{\mathcal{C}}(\sigma(E_i))=\sum_{j=1}^{k'}\lambda_i+\sum_{D\in\mathcal{F}}\lambda_Da_D$ is at least two. We conclude by Proposition~\ref{prop:canterm}.\\
 
The proof of the second assertion is easier.

Let $X$ be a Gorenstein $G/H$-embedding and let $f:\,V\longrightarrow X$ be a $G$-equivariant resolution of $X$. Let $E_i$ be an exceptional divisor of $f$. And let $(\mathcal{C},\mathcal{F})\in\mathbb{F}_X$ such that $\sigma(E_i)\in\mathcal{C}$ so that $h_{\mathcal{C}}(\sigma(E_i))$ is positive. Then, since $X$ is Gorenstein, $h_{\mathcal{C}}(\sigma(E_i))$ is a positive integer.
Hence, $a_i=h_{\mathcal{C}}(\sigma(E_i))-1$ is a non-negative integer. It implies that $X$ has canonical singularities.
\end{proof}

We can now conclude the paper by the following diagram and example.\\
\begin{figure}
\begin{center}
\caption{Relations between the singularities considered in this paper}\label{fig:relations}
\vspace{0.4cm}
\begin{tikzpicture}[scale=1]
\node [text width=3cm, text centered] (a) at (0,0) {Smooth};
\node [text width=3cm, text centered] (b) at (-3,-2) {Locally factorial};
\node [text width=3cm, text centered] (c) at (-5,-4) {$\Qbb$-factorial};
\node [text width=2cm, text centered] (d) at (-1.5,-4) {Gorenstein};
\node [text width=3cm, text centered] (e) at (-3,-6) {$\Qbb$-Gorenstein};
\node [text width=2.5cm, text centered] (f) at (4,-2) {terminal singularities};
\node [text width=2.5cm, text centered] (g) at (6,-4) {canonical singularities};
\node [text width=3cm, text centered] (h) at (4,-6) {log terminal singularities};

\draw [double distance=2pt,->] (a) -- (b);
\draw [double distance=2pt,->] (a) -- (f);
\draw [double distance=2pt,->] (b) -- (c);
\draw [double distance=2pt,->] (b) -- (d);
\draw [double distance=2pt,->] (f) -- (g);
\draw [double distance=2pt,->] (c) -- (e);
\draw [double distance=2pt,->] (d) -- (e);
\draw [double distance=2pt,->] (g) -- (h);
\draw [color=red,double distance=2pt,->] (d) -- node[above] {if with rational singularities} node[below] {in particular if spherical}(g);
\draw [color=red,double distance=2pt,->] (b) -- node[above] {if spherical} (f);
\draw [color=red,double distance=2pt,->] (e) -- node[above] {if spherical}(h);

\draw [double distance=2pt,<-] (e) to [bend right=30] (h);
\end{tikzpicture}
\end{center}
\end{figure}
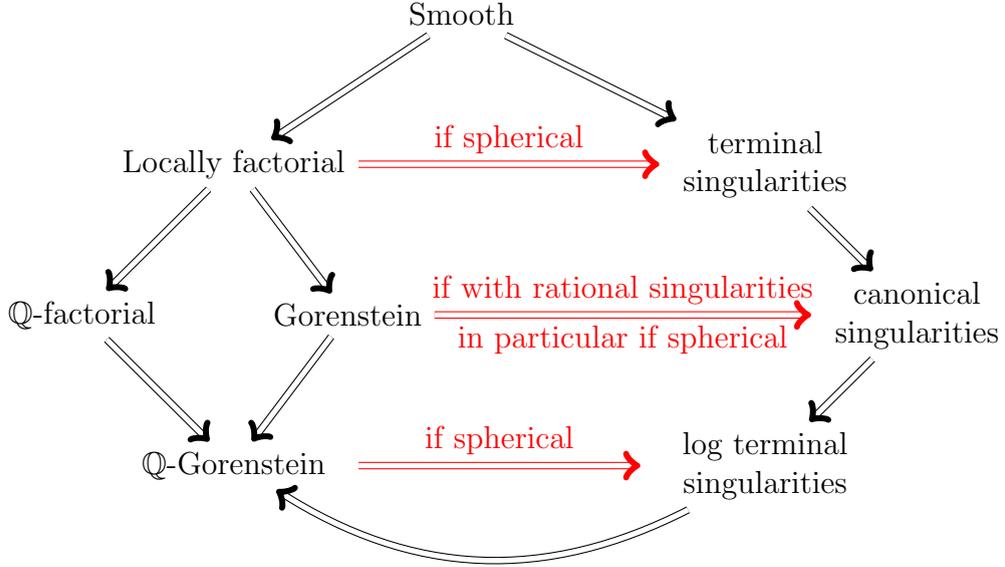

\begin{ex}\label{ex:singSL3U1}

Consider the horospherical homogeneous space ${\rm SL}_3/U$ where $U$ is the subgroup of upper triangular matrices with ones on the diagonal. In that case, with the notations above, $M$ is the set of characters of a maximal torus of ${\rm SL}_3$. In particular,  $M$ and $N\simeq \Zbb^2$. The valuation cone is $N_\Qbb$, $\mathcal{D}$ has two elements $D_\alpha$ and $D_\beta$ whose image by $\sigma$ are respectively the simple coroots $\alpha^\vee$ and $\beta^\vee$.

We only consider here complete colored fans. Note also that, here, the images of the colors are all distinct. Then, to represent a colored fan, we only draw the edges of the fan and we represent a color of the fan by bordering in grey the white circle corresponding to the image of the color of $G/H$.

Then we give in Figure~\ref{fig:singSL3U} a list a colored fans (corresponding to projective ${\rm SL}_3/U$-embed\-dings), by pointing those who are smooth, locally factorial, $\Qbb$-factorial or not $\Qbb$-factorial, Gorenstein, $\Qbb$-Gorenstein or not $\Qbb$-Gorenstein, with terminal or canonical singularities, or only with log terminal singularities. When the variety $X$ is not $\Qbb$-Gorenstein, we can also precise if there exists, or not, a $\Qbb$-divisor $D$ such that the pair $(X,D)$ is klt.

We only write the optimal singularities.

We also represent by arrows all $G$-equivariant morphisms between these $G/H$-embeddings.

In this example, we see in particular that there exist horospherical varieties with terminal singularities that are either not Gorenstein or not $\Qbb$-factorial. It means that we list all possible implications in Figure~\ref{fig:relations}.

\begin{figure}

\caption{Singularities of the ${\rm SL}_3/U$-embeddings of Example~\ref{ex:singSL3U1}}\label{fig:singSL3U}
\vspace{0.2cm}
\begin{tikzpicture}[scale=1]
\node [text width=3cm, text centered] (a) at (0,-1.5) {\scalebox{0.35}{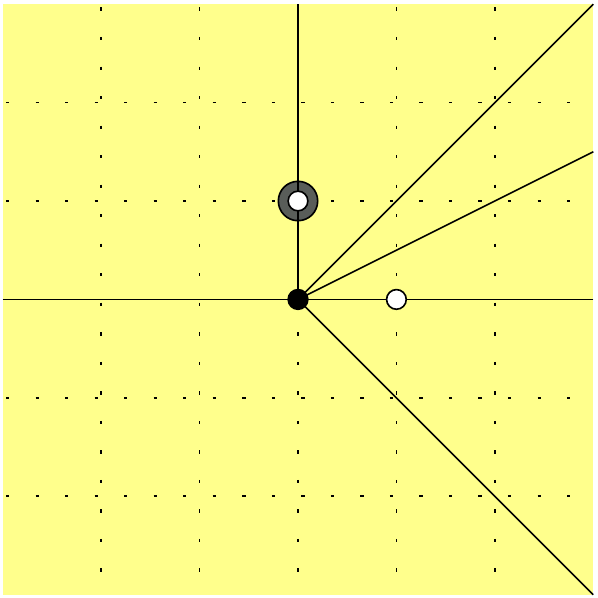}
\small{Smooth}
};

\node [text width=3cm, text centered] (b) at (-6,-5.5) {\scalebox{0.35}{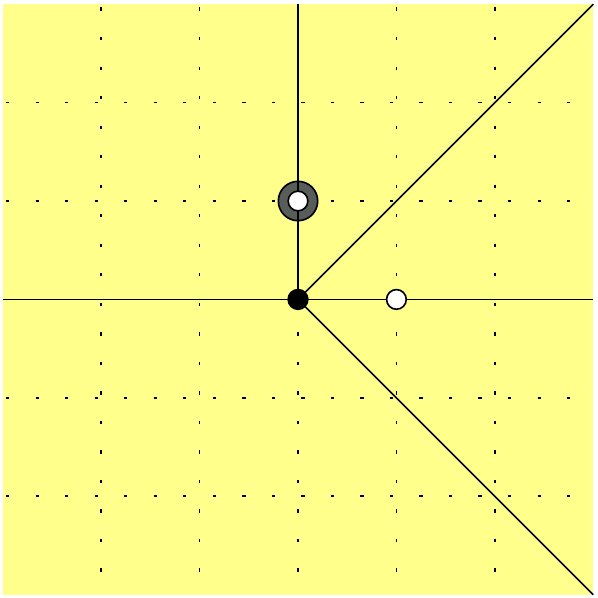}
\small{Smooth}
};
\node [text width=3.2cm, text centered] (c) at (-2,-5.5) {\scalebox{0.35}{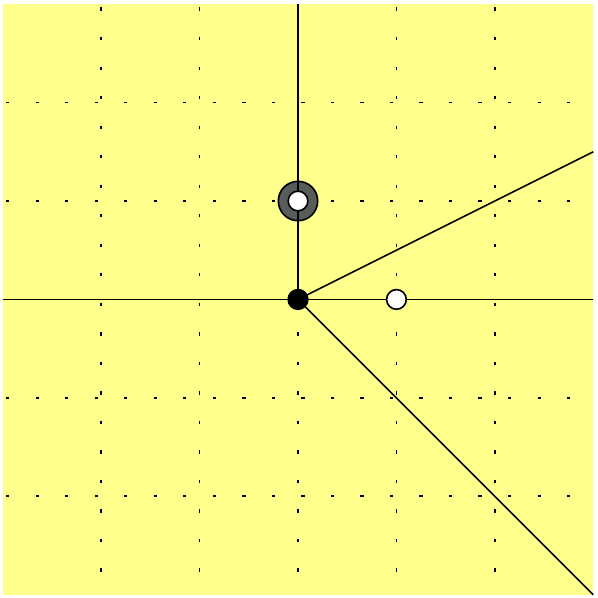}
\small{$\Qbb$-factorial, not Gorenstein, terminal singularities}
};
\node [text width=3cm, text centered] (d) at (2,-5.5) {\scalebox{0.35}{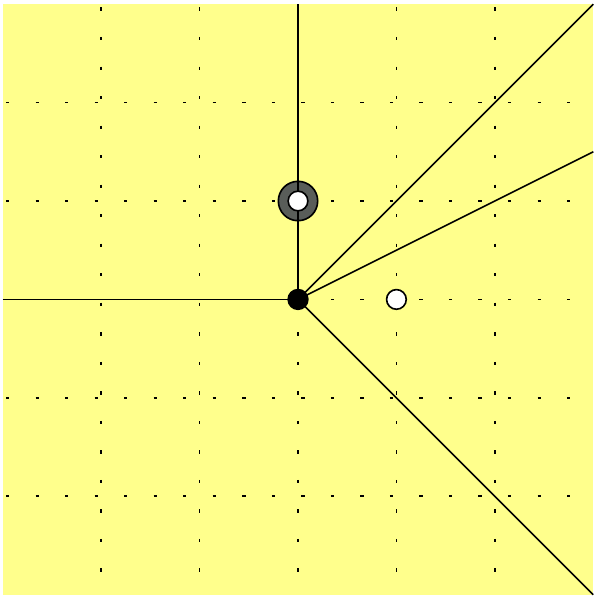}
\small{$\Qbb$-factorial, not Gorenstein (and log terminal singularities)}
};
\node [text width=3.2cm, text centered] (e) at (6,-5.5) {\scalebox{0.35}{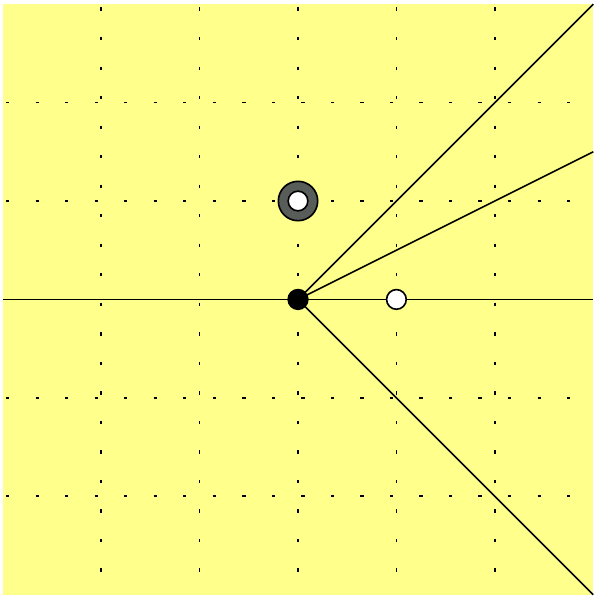}
\small{Not $\Qbb$-factorial, Gorenstein, terminal singularities}
};

\node [text width=3cm, text centered] (f) at (-6,-10) {\scalebox{0.35}{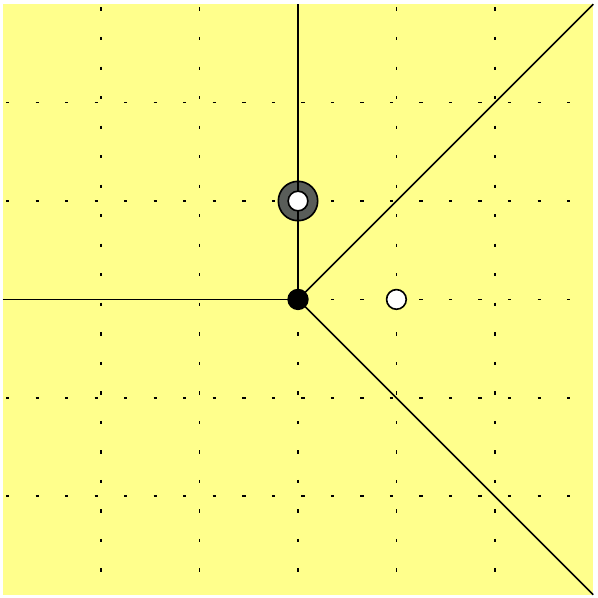}
\small{$\Qbb$-factorial, Gorenstein (and canonical singularities)}
};

\node [text width=3cm, text centered] (g) at (-2.5,-10) {\scalebox{0.35}{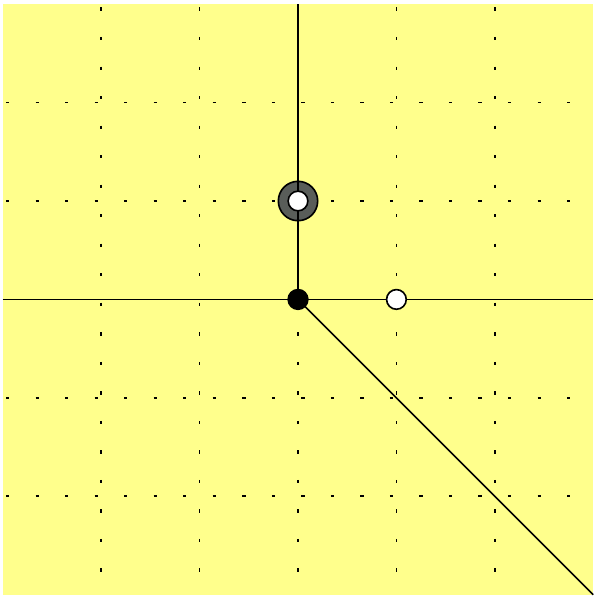}
\small{Smooth}
};

\node [text width=4cm, text centered] (h) at (2,-10) {~~~~\scalebox{0.35}{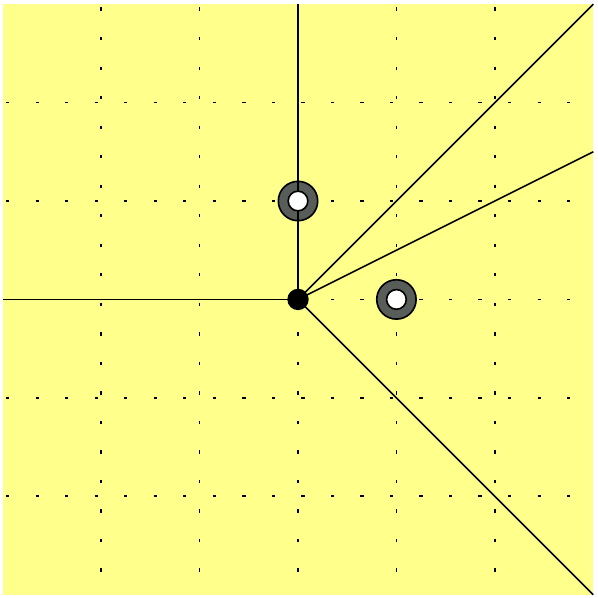}
\small{~~~~~Not $\Qbb$-Gorenstein,
there exists no klt pair}
};
\node [text width=3cm, text centered] (i) at (5.5,-10) {\scalebox{0.35}{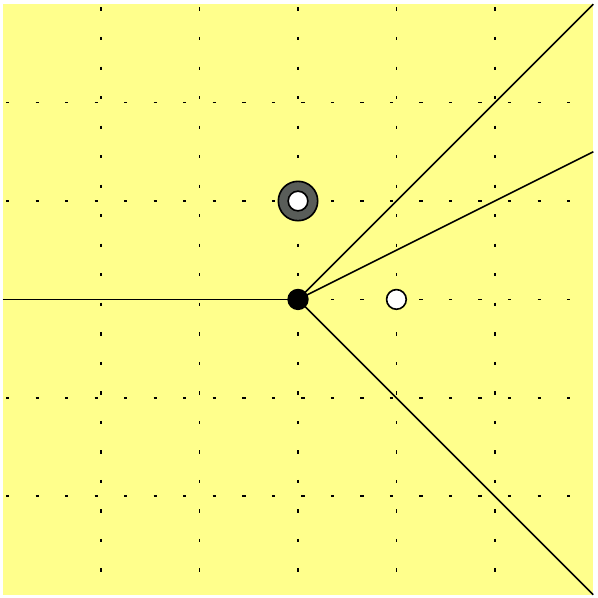}
\small{ $\Qbb$-Gorenstein (and log terminal singularities)}
};

\node [text width=3cm, text centered] (j) at (-2.5,-14) {\scalebox{0.35}{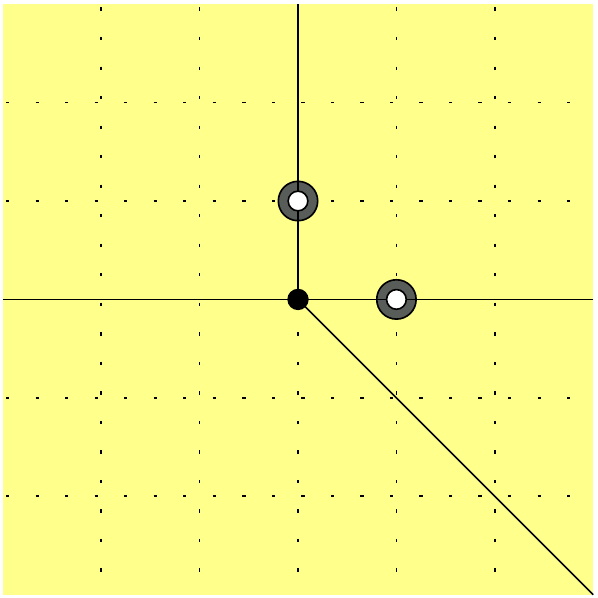}
\small{Locally factorial (and terminal singularities)}
};

\node [text width=3cm, text centered] (k) at (-6,-14.5) {\scalebox{0.35}{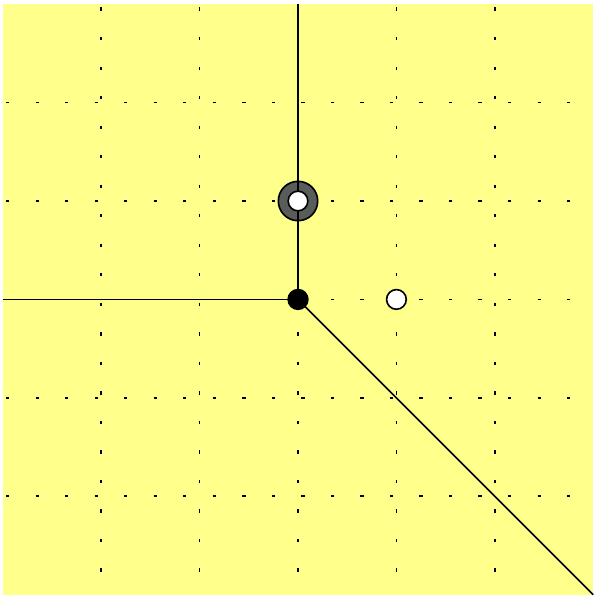}
\small{Smooth}
};

\node [text width=3.2cm, text centered](l) at (2.5,-14.2) {\scalebox{0.35}{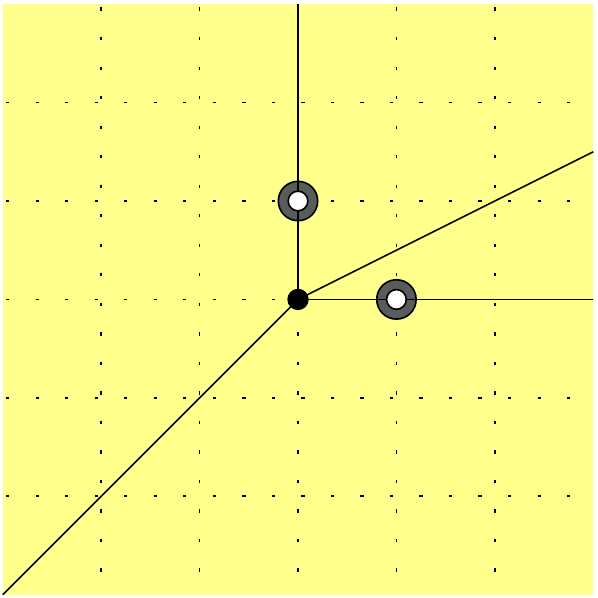}
\small{$\Qbb$-factorial, not Gorenstein, terminal singularities}
};

\node [text width=3cm, text centered](o) at (6,-14.5) {\scalebox{0.35}{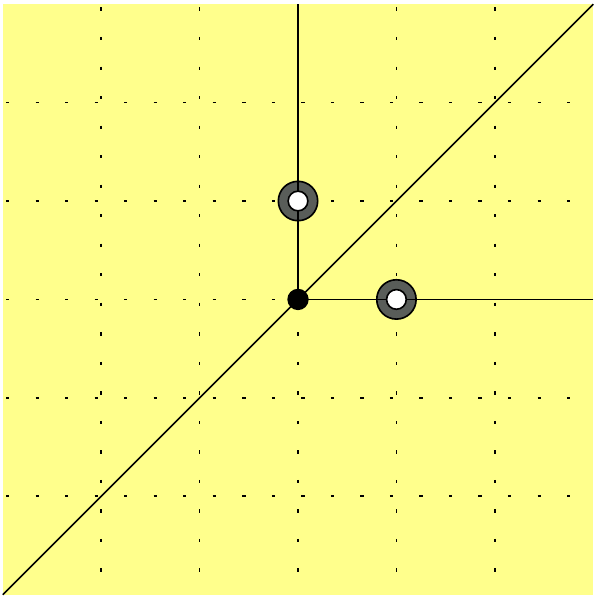}
\small{Smooth}
};

\node [text width=4cm, text centered](m) at (-4,-18.5) {~~~\scalebox{0.35}{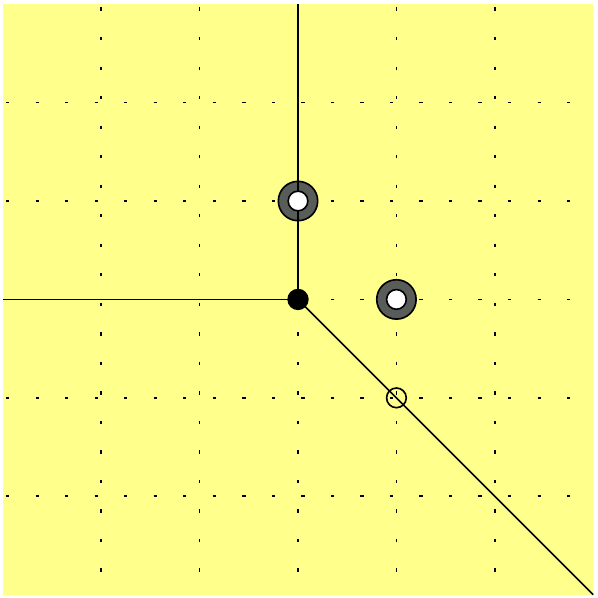}
\small{~~~~Not $\Qbb$-Gorenstein, there exists klt pairs (for example with $D=\frac{1}{2}X_5+\frac{1}{2}D_\beta$)}
};

\node [text width=3cm, text centered](n) at (2.5,-18.6) {\scalebox{0.35}{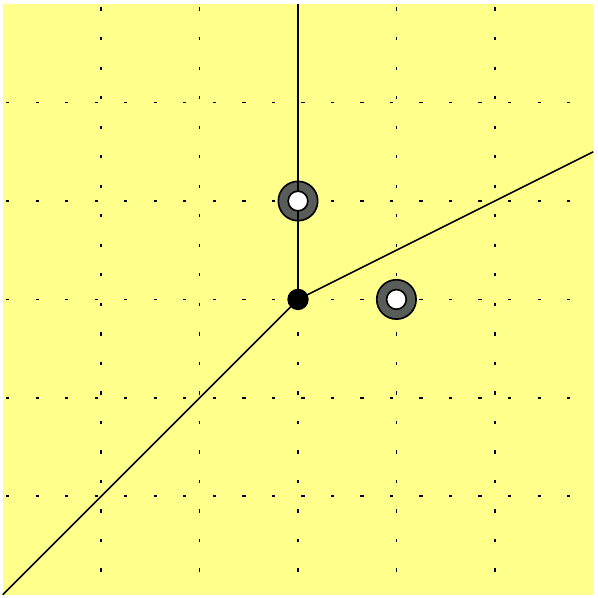}
\small{$\Qbb$-Gorenstein, terminal singularities}
};

\node [text width=3cm, text centered](p) at (6,-18.5) {\scalebox{0.35}{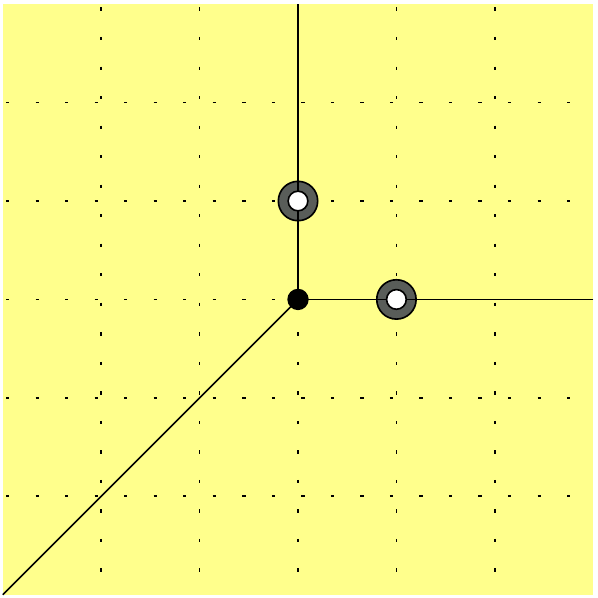}
\small{Locally factorial (and terminal singularities)}
};

\draw[->,>=latex] (a) to[out=-165,in=75] (b);
\draw[->,>=latex] (a) -- (c);
\draw[->,>=latex] (a) -- (d);
\draw[->,>=latex] (a) to[out=-15,in=105] (e);
\draw[->,>=latex] (b) -- (f);
\draw[->,>=latex] (b) -- (g);
\draw[->,>=latex] (c) --   (g);
\draw[->,>=latex] (d) --   (h);
\draw[->,>=latex] (d) to[out=-120,in=40]  (f);
\draw[->,>=latex] (d) --   (i);

\draw[->,>=latex] (e) --  (i);
\draw[->,>=latex] (f) --  (k);
\draw[->,>=latex] (g) --   (k);
\draw[->,>=latex] (g) --   (j);
\draw[->,>=latex] (j) -- (m);
\draw[->,>=latex] (k) -- (m);
\draw[->,>=latex] (l) -- (n);
\draw[->,>=latex] (l) -- (p);
\draw[->,>=latex] (h) to[out=-130,in=20] (m);
\draw[->,>=latex] (o) -- (p);

\end{tikzpicture}

\end{figure}
\end{ex}

\newpage
\bibliographystyle{amsalpha}
\bibliography{SingSpher}
\bigskip\noindent

\end{document}